\documentclass[12pt,a4paper,twoside]{article}

\usepackage{a4wide, amssymb, amsmath, amsthm, graphics, comment, xspace, enumerate}
\usepackage{graphicx}
\usepackage[a4paper,colorlinks=true,citecolor=blue,urlcolor=blue,linkcolor=blue,bookmarksopen=true,unicode=true,
          pdffitwindow=true]{hyperref}
\usepackage{algorithmic, algorithm}

\newcommand*\samethanks[1][\value{footnote}]{\footnotemark[#1]}

\newtheorem{te}{Theorem}[section]
\newtheorem{pro}{Proposition}[section]

\newtheorem{lemma}{Lemma}[section]

\newtheorem{problem}{Problem}[section]
\newtheorem{conjecture}{Conjecture}[section]

\newcommand{\beq}{\begin{eqnarray}}
\newcommand{\eeq}{\end{eqnarray}}

\newcommand{\beqs}{\begin{eqnarray*}}
\newcommand{\eeqs}{\end{eqnarray*}}

\newcommand{\inst}[1]{$^{#1}$}

\begin{document}

\title{Sandwiching  saturation number of fullerene graphs\thanks{Partly supported by Slovak-Slovenian grant no.~SK-SI-0005-10 from the Slovak Research and Development Agency, and France-Slovenian grant BI-FR/14-15-PROTEUS-001 }}

\author{Vesna~Andova\inst{1}\thanks{Partially supported by Slovenian ARRS Program P1-00383 and Creative Core - FISNM - 3330-13-500033.},\quad Franti\v sek Kardo\v s\inst{2}\thanks{Partly supported by the French ANR project DORSO.}
,\quad
    Riste~\v{S}krekovski\inst{3}\samethanks[1]
    }

\maketitle
\begin{center}
{\footnotesize
\inst{1} Faculty of Electrical Engineering  and Information Technologies, Ss Cyril and Methodius Univ.\\
          Ruger Boskovik bb, 1000 Skopje, Macedonia\\
          e-mail:{\tt vesna.andova@gmail.com}\\

\inst{2} LaBRI, University of Bordeaux  \\ 351, cours de la Lib\'eration, 33405 Talence \\ e-mail: {\tt frantisek.kardos@labri.fr}\\

\inst{3} Department of Mathematics, University of Ljubljana, 1000 Ljubljana \\
\& Faculty of Information Studies, 8000 Novo Mesto\& \\ Faculty of Mathematics, Natural Sciences and Information Technologies, University of Primorska,   Koper, Slovenia.\\

         e-mail: {\tt skrekovski@gmail.com}\\
}
\end{center}

\maketitle
\begin{center}
\end{center}

%
%
%
%
%
\begin{abstract}
    The saturation number of a graph $G$ is the cardinality of any smallest maximal matching of $G$, and it is denoted by $s(G)$. Fullerene graphs are cubic planar graphs with exactly twelve 5-faces; all the other faces are hexagons. They are used to capture the structure of carbon molecules. Here we show that the saturation number of fullerenes on $n$ vertices is essentially $ n/3$.

\end{abstract}

{\small \hspace{0.25cm} \textbf{Keywords:} fullerene graph, saturation number}

%
%
%
%
%
%

\section{Introduction}
Fullerenes are polyhedral molecules made entirely of carbon atoms.
The first fullerene, $C_{60}$, also known as buckminsterfullerene, was discovered in 1985~\cite{KHOCS}.
The name was a homage to Richard Buckminster Fuller, whose geodetic domes it resembles.
Due to the wide specter of possible applications, fullerenes attract the  attention of diverse research communities.
One of the main driving forces behind that work has been a desire to identify structural properties characteristic for stable fullerenes.
Fullerenes can also be represented as graphs; each atom is a vertex, and the bonds between them are the edges of the graph.
The methods of graph theory applied to the mathematical models of fullerene molecules resulted with a number of graph-theoretical invariants as potential stability predictors~\cite{D, FL}. Although we have the answerers to a  lot of problems and questions for fullerenes and their stability, still there is much more to be done~\cite{FRFHC, M}. For more results
about fullerenes, their mathematical, physical or chemical properties, see~\cite{FM}.

One very important question, a question that receives a lot of attention, is the fullerenes stability. The aim is finding a graph theoretical invariant(s) closely related to the stability of fullerene molecules. Number of different graph invariants that correlate with the stability were
studied. Among those invariant is the \emph{saturation number} $s(G)$ -- the cardinality of the smallest maximal matching  in a fullerene graph $G$.
The smallest maximal matching in a graph is also known as the \emph{smallest independent edge dominating set}.


Clearly, the set of vertices that is not covered by a maximal matching is independent~\cite{E}. (A set of vertices $I$ is \emph{independent} if no two  vertices from $I$ are adjacent.) This observation provides an obvious lower bound on saturation number of the graph $G$, i.e.~$(n-|I|)/2\le s(G)$, where $G$ is graph of order $n$.

The saturation number has another meaning for the chemists: it corresponds to the smallest possible number of large substituents/addents (those that occupy two adjacent atoms) that saturate the molecule. Independent set is another parameter of similar meaning: it is the maximum possible number of addents that cannot be attached to adjacent atoms.
Besides in chemistry the saturation number (smallest independent edge dominating set) has a list of interesting applications in engineering, networks, etc.

The saturation number of fullerene graph was studied in~\cite{ADKLS, D3}, where the following bounds were established.

\begin{te}\label{te:sat-diam}
For any fullerene graph $G$ on $n$ vertices and diameter $d$, it holds
        $$
        \frac{3n}{10}\le s(G)\le \frac n2-\frac 14(d - 2).
        $$
In particular,
        $$
        s(G)\le\frac n2-\frac{\sqrt{24n-15}-15}{24}.
        $$
\end{te}
The saturation number and independent sets on fullerene graphs and triangle-free cubic planar graphs are also studied in~\cite{FKS,HT}.
The lower bound in the previous theorem relies only on the 3-regularity of fullerenes. That makes us believe that this bound is not that precise.
%

Zito~\cite{Z} provided a probabilistic argument that almost all cubic graphs on $n$ vertices have saturation number at least $0.315812\,n\,$. On
the other hand there are at least two fullerene graph on $n$ vertices whose saturation number
is $ 3n/10$. Those graphs are dodecahedron and buckminsterfullerene. These
two fullerenes are the only fullerenes with icosahedral symmetry whose saturation number
satisfies the lower bound of Theorem~\ref{te:sat-diam}~\cite{D3}.

In this paper we show that the saturation number of fullerenes on $n$ vertices is essentially $n/3$.
\section{Definitions and preliminaries}

A \emph{fullerene graph} is a 3-connected 3-regular planar graph with only pentagonal and hexagonal faces. Owing to the Euler formula there are exactly 12 pentagons, but there is no restriction on the number of hexagons. Gr\"{u}nbaum and Motzkin~\cite{GM} showed that fullerene graphs on
$n$ vertices exist for all even $n \ge 24$ and for $n = 20$, i.e., there exists a fullerene graph with $\alpha$ hexagons where $\alpha$ is any integer distinct from 1. Although the number of pentagonal faces is negligible compared to the number of hexagonal faces, their layout is crucial for the shape of a fullerene graph. If all pentagonal faces are equally distributed, we obtain fullerene graphs of spherical shape with icosahedral symmetry, whose smallest representative is the dodecahedron. On the other hand, there is a class of fullerene graphs of tubular shapes, called \emph{nanotubes}.

A {\it patch}  is a 2-connected plane graph with only pentagonal or hexagonal faces, except maybe one face -- the outer face; all interior vertices are of degree 3, and all vertices incident to the outer face (on the \emph{boundary} of the patch), are of degree 2 or 3. A patch with no pentagons is called  a {\it hexagonal} patch. Note that by cutting along a cycle in a fullerene graph we always obtain two patches.

Let the number of vertices of degree 2 incident to a face $B$ be denoted by $n_2(B)$. Similarly, let $n_3(B)$ denotes the number of vertices of degree 3 incident to the face $B$.

In \cite{KS} following lemma  is proven.
\begin{lemma}\label{lemma:patch}
    Let $G$ be a patch with $p$ pentagons, and an outer face $B$. Then, $$n_2(B)-n_3(B)=6-p\,.$$
\end{lemma}

\noindent From the previous lemma it follows that a patch $G$ has 6 pentagons if and only if $n_2(B)=n_3(B)$.

An {\it infinite (hexagonal) tube} is obtained from a planar hexagonal grid by identifying objects (vertices, edges, faces) lying on two parallel lines. The
way the grid is wrapped is represented by a pair of integers $(p_1, p_2)$. The numbers $p_1$ and $p_2$ denote the coefficients of the linear combination of the unit vectors $\vec{a}_1$ and $\vec{a}_2$ such that the vector $p_1\vec{a}_1 + p_2\vec{a}_2$ joins pairs of identified points. 
We can always assume that $p_1 \ge  p_2$ since we want to avoid the mirror effect.
Figure~\ref{fig:tube} shows  the  construction of a $(4,3)$- infinite (hexagonal) tube. 

\begin{figure}[htb]
\begin{center}
		      \includegraphics[scale=0.8]{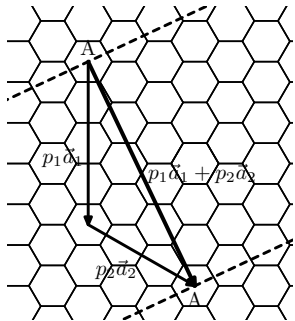}
\end{center}
\caption{Construction of  the cylindrical part of a $(4,3)$-nanotube. The hexagons with the same name overlap.}
\label{fig:tube}
\end{figure}




Denote by $\mathrm{Tran}_{\vec{a}}A$   a translation of an object $A$ for a vector $\vec{a}$.
Let $h_0$ be a hexagon of the infinite (hexagonal) tube. The set $\{h_0,h_1,\dots,h_{p_1+p_2-1}\}$ of hexagons with $h_i = \mathrm{Tran}_{\vec{b_i}}h_0$, where $\vec{b_i} = i\vec{a_1}$
for $0\le i \le p_1$ and $\vec{b_i} = p_1\vec{a_1} + (i - p_1)\vec{a_2}$ for $p_1 < i \le p_1 + p_2$, is called \emph{characterizing ring}.
From the definition of the $(p_1, p_2)$ infinite hexagonal tube, it follows that hexagons $h_0$ and $h_{p_1+p_2}$ are overlapping, thus identified, hence $h_0 = h_{p_1+p_2}$.  Notice that the characterizing ring can be defined differently as shown on Figure~\ref{fig:ring} -- the order of $p_1$ occurrences of the vector $\vec{a}_1$, and $p_2$ occurrences of the vector $\vec{a}_2$ can be arbitrary.

An infinite tube (defined earlier) can be considered as an union of consecutive characterizing rings $R_i$ with $i\in \mathbb{Z}$.
Let a {\it tube}  be a subgraph of an infinite (hexagonal) tube build by a finite number of consecutive characterizing rings.

\begin{figure}[htp!]
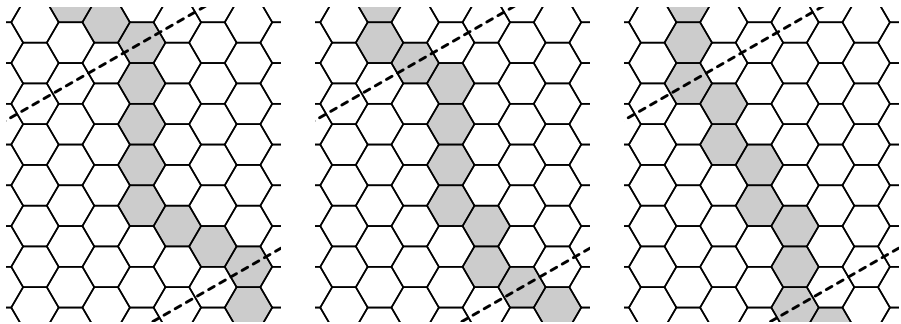

	\centerline{
		\begin{tabular}{ccc}
             \includegraphics[scale=0.8]{rings.1}&
             \includegraphics[scale=0.8]{rings.2}&
             \includegraphics[scale=0.8]{rings.3}			
		\end{tabular}}
	\caption{All possible types of characterizing rings for a $(4,3)$-infinite (hexagonal) tube. The shaded hexagons on the dashed lines are overlapping. }
	\label{fig:ring}
\end{figure}

\emph{Nanotubical graphs} or simply \emph{nanotubes} are fullerene graphs decomposable into a cylindrical part (a tube) and two patches (the \emph{caps}) containing six pentagons each. The cylindrical part of the nanotube is a subgraph of a $(p_1, p_2)$-infinite hexagonal tube for certain $(p_1,p_2)$, and therefore the nanotube fullerene is also called a $(p_1, p_2)$-nanotube.
In that fashion Figure~\ref{fig:tube} also shows  the construction of the tube of a $(4,3)$-nanotube.


Among the possible stability predictors and molecular descriptors for fullerene graphs there are invariants like the number of perfect matchings~\cite{KKMS, LP, EKKKN}, bipartite edge frustration, the independence number, the saturation number, etc. A \emph{bipartite edge frustration} of a graph $G$, denoted by $\varphi(G)$, is the smallest cardinality of a subset of $E(G)$ that needs to be removed from $G$ in order to obtain a bipartite spanning subgraph.
A set $I\subseteq V (G)$ is \emph{independent} if no two vertices from $I$ are adjacent in $G$. The cardinality
of any largest independent set in $G$ is called the \emph{independence number} of $G$, and it is denoted by $\alpha(G)$. For more results concerning independence number and bipartite edge frustration
see~\cite{D}.

The saturation number is a structural invariant directly related to matchings. A \emph{matching} in $G$ is a collection $M$ of edges of $G$ such that no two edges from $M$ have a vertex in common. If a matching $M$ covers all vertices of $G$ we say that $M$ is a \emph{perfect matching}. The perfect matching in cubic, and therefore in fullerene graphs was studying in number of different papers. At first L\'{o}vasz-Plummer~\cite{LP} conjectured that there is a $c > 0$ such that every bridgeless cubic graph has at least $2^{cn}$ perfect matchings, where $n$ is the
number of nodes of the graph. This conjecture was proven by Esperet et al.~\cite{EKKKN}. The problem of perfect matching was also studied by Chudovsky and Seymour~\cite{CS}

A  matching that cannot be improved by adding an edge is called a \emph{maximal matching}.
The \emph{saturation number} of $G$ is the cardinality of any smallest maximal matching of $G$.  We say that each edge {\it dominates} its adjacent edges.  A{ {\it independent edge dominating set} is a dominating set in which no two edges are adjacent. The cardinality of the independent dominating set  of a graph $G$ is its saturation number $s(G)$.  Finding the independent dominating set of a graph is an NP-hard problem~\cite{YG}. 


\section{Upper bound on the saturation number}

In this section we improve the upper bound on the saturation number in fullerene graphs. We describe a construction to find a maximal matching of size $n/3 + C\sqrt{n}$, where $C$ is a constant.

First we define a maximal matching on an infinite (hexagonal) tube of type $(p_1,p_2)$, $p_1\ge p_2$ and $p_1>0$.
\begin{pro}
\label{pro:inftube}
    There is a maximal matching $M_0$ on any infinite tube $G_0$ such that from each hexagon precisely two vertices are not covered by $M_0$.
\end{pro}
\begin{proof}
    We provide a construction of a maximal matching $M_0$ of the infinite tube $G_0$ of the type $(p_1,p_2)$, $p_1\ge p_2$, $p_1>0$. 
    We call the edges of $M_0$ \emph{black} edges; we also call the vertices covered by $M_0$ \emph{black} vertices. The vertices not covered by $M_0$ form an independent set; we call them \emph{white} vertices.

    For each hexagon $h_i$ we call the common edges with the adjacent hexagons in the direction $\vec{a}_1$ and $-\vec{a}_1$, an {\it $a_1$-edge} and {\it  $-a_1$-edge}, respectively. Similarly, we name the common edges with the adjacent hexagons in direction $\vec{a}_2$ and $-\vec{a}_2$ (see Figure~\ref{fig:Ring}$(a)$).

    We choose a  characterizing ring $R_1:\, h_0,h_1,\dots,h_{p_1+p_2-1},h_{p_1+p_2}=h_0$ of $G_0$. For each hexagon $h_i$ we color the $a_1$-edge black; white vertex will be the vertex  incident to a $a_2$- or $-a_2$-edge  which is not black yet, see Figure~\ref{fig:Ring}$(a)$. For the hexagons in the next characterizing ring $R_2$ we propagate similar pattern; for each hexagon we color black the $a_2$-edge. White vertices are all the remaining vertices of the ring that are not colored black yet, see Figure~\ref{fig:Ring}$(b)$.

\begin{figure}[htp!]
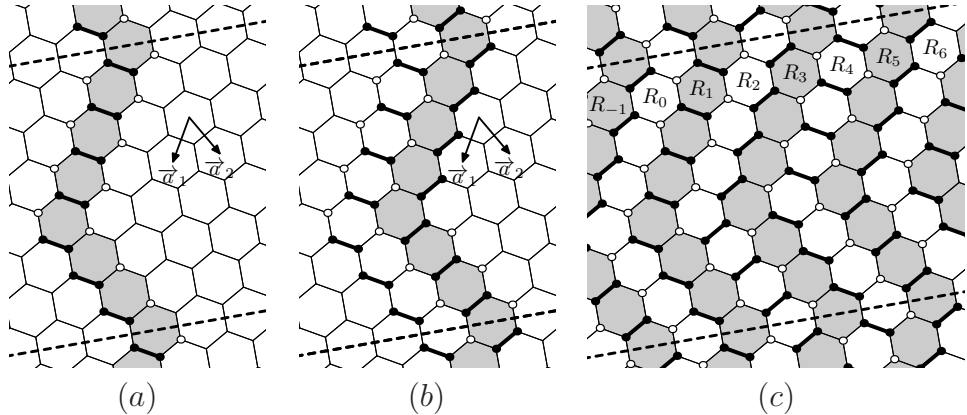

\label{fig:Ring}
	\begin{center}
			\begin{tabular}{ccc}
    \includegraphics[scale=0.8]{saturation.1}\qquad&
    \includegraphics[scale=0.8]{saturation.2}&
    \includegraphics[scale=0.8]{saturation.3}\\
$(a)$& $(b)$& $(c)$ 
			\end{tabular}
    \end{center}
	\caption{Defining a matching on a $(4,3)$ tube. The dashed lines are overlapping, $a_1$ and $a_2$ edges, named after the corresponding vectors (green and red respectively). $(a)$ Defining a matching on the first ring (shaded). $(b)$  Defining a matching on the second ring (shaded). $(c)$ Maximal matching covering exactly four vertices from each hexagon. The characterizing rings are alternately white and shaded. }
\end{figure}

 For a moment we skip defining a matching on the characterizing ring $R_3$. Instead we define a matching on the next ring $R_4$, in the same fashion as the matching on $R_1$. Now, the matching on $R_2$ and $R_4$ defines the matching on $R_3$. We extend this matching to the tube in the following way: the matching on the ring $R_k$, $k\in \mathbb{Z}$ is defined in the same way as the matching on the characterizing ring $R_j$, $j=0,1,2$ if $k\equiv j\, ({\rm  mod }\, 3)$.


    This way we obtain a desired maximal matching on an infinite (hexagonal) tube. See Figure~\ref{fig:Ring}$(c)$ for illustration.
\end{proof}


In~\cite{DLS}, Dvo\v r\'ak et al.~were investigating how many edges must be removed from a fullerene graph such that the
new graph is bipartite. It is clear that from each pentagon an edge must be removed, but that still does not give a bipartite graph since odd cycles still exist.
They found out that at most $O(\sqrt{n})$ edge must be removed in order to get a bipartite graph.

\begin{te}\label{Thm:bipart}\cite{DLS}
    If $F$ is a fullerene graph with $n$ vertices, then $\varphi (F) = O(\sqrt n)$.
\end{te}

Even more, in the same paper they showed that for every pentagon in a fullerene graph, there exist at least five other pentagons at total distance $O(\sqrt n)$.

\begin{lemma}[Six Pentagons Lemma]\cite{DLS}
\label{c.10}
For every pentagonal face $f$ in a fullerene graph $F$ with $n$ vertices, there exist at least five other pentagonal faces whose distance to $f$ in the dual $F^*$ is at most $\sqrt{63n/2}+14$.
\end{lemma}

Using this result we prove the following.

\begin{te}
    Let $F$ be a fullerene graph on $n$ vertices. Then
    $$
    s(F)\le \frac n3 + O(\sqrt{n}).
    $$
\end{te}

\begin{proof}
Let $F$ be a fullerene graph. On the set of the twelve pentagonal faces of $F$ we consider the transitive closure $\sim$ of the relation "the distance between $f_1$ and $f_2$ in the dual is at most $\sqrt{63n/2}+14$". By Six Pentagons Lemma, for each pentagonal face $f$ there are at least five other pentagonal faces $f^\prime$ such that $f\sim f^\prime$. Therefore, depending on the fullerene structure, two cases are possible:
\begin{itemize}
    \item[(A)] \textit{There are two equivalent classes with respect to $\sim$;}
    \item[(B)] \textit{There is just one class with respect to $\sim$ containing all the twelve  pentagons of $F$.}
\end{itemize}
We first prove the theorem for case (A), and later we consider the case (B).
\begin{itemize}
\item[(A)] \textit{There are two equivalent classes with respect to $\sim$.} This means there are two 6-tuples of pentagons ``far away" from each other.
    In this case we find two trees $T^*_1$ and $T^*_2$ in the dual $F^*$ of $F$, each covering the corresponding six pentagonal faces. Such a tree always exists: it suffices to  choose one vertex of degree 5, and using breadth-first search find the shortest paths to the other five vertices of degree 5. The union  of these is a desired tree in $F^*$. Let $T_i$ be the set of edges in $F$ corresponding to the edges of $F^*$ with both endvertices in $T^*_i$, $i=1,2$. Observe that $T_i$ may contain (a bounded number of) edges corresponding to edges in $F^*$ which are not edges of the tree $T_i^*$. Among all the trees possible, for $T_i^*$ we choose one with as few edges in $T_i$ as possible. The overall number of vertices of $T_1^*\cup T_2^*$ by Lemma~\ref{c.10} is at most $10(\sqrt{63n/2}+14)$.

    Let $Q_i$ be the graph obtained as a union of the boundary cycles for the faces of $F$ corresponding to the vertices of $T_i^*$. As $T_i^*$ covers exactly 6 vertices of degree 5, $Q_i$ is a fullerene patch containing exactly 6 pentagons. Let $C_i$ be the boundary cycle of $Q_i$ (the binary sum of the boundaries of its faces). Observe that $C_i$ is connected since $T_i^*$ is a tree in $F^*$.

On the other hand, it is easy to see that each vertex of $F$ is incident to 0, 1 or 3 edges in $T_i$ (otherwise there would be two vertices in $T_i^*$ joined by an edge not corresponding to an edge of $T_i$). Let $R_i$ be the set of vertices incident to 3 edges in $T_i$. The graph $P_i=F - ( T_i\cup R_i)$ is another fullerene patch containing exactly 6 pentagons. Clearly, $P_i\cap Q_i=C_i$. See Figure \ref{fig:A} for an illustration of $T_i$, $C_i$, $Q_i$ and $P_i$.
Finally, let $G=F - (T_1\cup R_1) - (T_2\cup R_2)$ be the subgraph of $F$ not containing any pentagons. Clearly all the faces of $G$ but two are hexagons.

\begin{figure}[htp!]
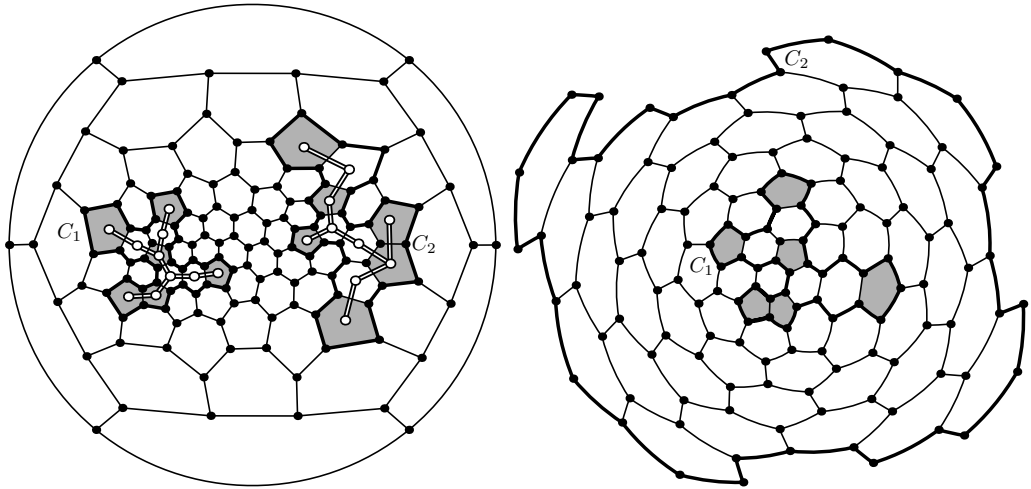

\centerline{
\includegraphics[scale=0.8]{A.0}
\includegraphics[scale=0.8]{Atube.1}
}
\caption{An example of a fullerene graph in which the pentagonal faces form two clusters of six. The boundary of the union of faces corresponding to the dual tree $T_i$ is the cycle $C_i$,  $i=1,2$ (left). The patch $P_2$ contains the patch $Q_1$ (right).}
\label{fig:A}
\end{figure}

For each vertex $v$ of $C_i$, either $d_{P_i}(v)=2$ and $d_{Q_i}(v)=3$ (if $v$ is incident to an edge of $T_i$), or $d_{P_i}(v)=3$ and $d_{Q_i}(v)=2$ (if $v$ is not incident to any edge of $T_i$). Let $n_{i,2}$ and $n_{i,3}$ be the numbers of vertices from $C_i$ of degree 2 and 3 in $Q_i$. Since both $P_i$ and $Q_i$ have exactly six pentagons, by Lemma \ref{lemma:patch} we have $n_{i,2}=n_{i,3}$.

The patch (with six pentagons) $Q_i$ is a cap of some nanotube. The type of the nanotube having $Q_i$ as a cap can be determined in the following way: Let $v_1,v_2,\dots,v_k$ be the vertices of $C_i$ in a cyclic order ($v_{k+1}=v_1$, $v_0=v_k$). If $d_{Q_i}(v_j)=2$, then the edges $v_{j-1}v_j$ and $v_{j}v_{j+1}$ are incident to two different added hexagons; towards the tube they form a 240 degree angle. If $d_{Q_i}(v_j)=3$, then the edges $v_{j-1}v_j$ and $v_{j}v_{j+1}$ are incident to the same added hexagon; they form a 120 degree angle. Informally speaking, for each vertex, the direction difference between the vectors $v_{j-1}v_j$ and $v_{j}v_{j+1}$ is either a `left turn' or a `right turn'.


If we choose the first edge $v_1v_2$ on the infinite hexagonal grid, the sequence of degrees of vertices of $C_i$ fully determines the position of all the other vertices. Since there is the same number of vertices of degree 2 and 3 on $Q_i$, there is the same number of left and right turns, so the edges $v_0v_1$ and $v_kv_{k+1}$ are equally oriented. Since on the tube $v_0=v_k$ and $v_{k+1}=v_1$, the difference $v_k-v_0=v_{k+1}-v_1$  on the infinite hexagonal grid determines the characterizing vector of the tube, see Figure \ref{fig:Atube} for illustration.

\begin{figure}[htp!]
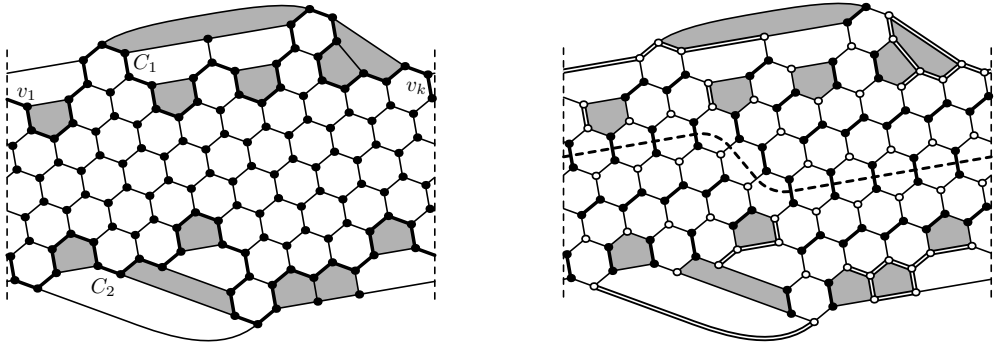

\centerline{
\includegraphics[scale=0.8]{A.2}
\hfil
\includegraphics[scale=0.8]{A.3}
}
\caption{The graph from Figure \ref{fig:A} drawn in such a way that the part containing only hexagons separating the two patches $Q_1$ and $Q_2$ containing pentagons is embedded into a nanotube (left). The difference $v_k-v_0$ on an infinite hexagonal grid determines the type of the nanotube. The graph from Figure \ref{fig:A} is a $(9,2)$-nanotube. The matching $M_1$ inherited from the indefinite hexagonal tube (right). The dashed line represents a characterizing ring.}
\label{fig:Atube}
\end{figure}

Since $C_1$ is the boundary of both $P_1$ and $Q_1$, and the type of the nanotube is determined solely by $C_1$, they can be considered as two different caps of the same nanotube.

Since $Q_1\subset P_2$ (and $Q_2\subset P_1$), after removing sufficiently large number of hexagons  from $P_2$ (resp. $P_1$) one can obtain $Q_1$ (resp. $Q_2$). The symmetric difference $G=F-(T_1\cup R_1)-( T_2\cup R_2)$ does not contain any pentagons, and therefore $Q_1$ is a cap for the same nanotube as $P_2$ ($Q_2$ as for $P_1$).



Once the type of the tube is determined, we embed $G$ into infinite tube $G_0$ with the predefined matching. By Proposition~\ref{pro:inftube}, there is a maximal matching $M_0$ on $G_0$ such that from each hexagon precisely two vertices are white (not covered by  the matching $M_0$). All the vertices of $G$ inherit the colors from the corresponding vertices of $G_0$.

Let $n_B$ and $n_W$ be the numbers of black and white vertices in $G$, respectively, let $r=|R_1|+|R_2|$.  Then $n=n_B+n_W+r$. Let $b_2$ and $b_3$ ($w_2$ and $w_3$) be the numbers of black (white, respectively) vertices of degree 2 and 3 incident to one of the exceptional faces of $G$. Since both $Q_1$ and $Q_2$ contain exactly 6 pentagons, we have
\begin{equation}
b_2+w_2=b_3+w_3.
\label{eq:BW}
\end{equation}
Let $h$ be the number of hexagonal faces of $G$, let $t$ be the  total number of face s of $Q_1$ and $Q_2$.
Then by double counting the vertices on the patches $Q_1$ and $Q_2$ we get
\begin{equation*}
6t-12=3r+2b_2+b_3+2w_2+w_3,
\end{equation*}
which combined with (\ref{eq:BW}) gives
\begin{equation}
2t-4=r+b_2+w_2.
\label{eq:t}
\end{equation}
Then, by double counting the face-vertex incidences in $G$ we get
\begin{equation*}
3(n_B-b_2-b_3)+b_2+2b_3=4h \quad\textrm{and}\quad
3(n_W-w_2-w_3)+w_2+2w_3=2h,
\end{equation*}
which together with (\ref{eq:BW}) implies
\begin{equation*}
3n_B=2n-2r+b_2-w_2-w_3.
\label{eq:nB}
\end{equation*}

Some of the edges of the matching $M_0$ defined on the infinite tube $G_0$ can have only one endvertex in $G$ and the other one not. This can only happen for black vertices of degree 2 in $G$; let $b_2^\prime$ be the number of them. We recolor those vertices white temporarily. Observe that for each such vertex, its two neighbors in $G$ are both incident to the same exceptional face. Let $M_1$ be the matching of $F$ obtained this way. Clearly, $|M_1|=(n_B-b_2^\prime)/2$. See Figure \ref{fig:Atube} for illustration.

The matching $M_1$ is not necessarily a maximal matching of $F$, however, two white vertices can only be adjacent in $F$ if they are both incident to the same exceptional face of $G$. We make the matching $M_1$ maximal by adding an arbitrary maximal matching of the subgraph of  $F$ induced by the white vertices incident to the two exceptional faces of $G$ and the vertices from $R_1\cup R_2$. This way we obtain a matching $M$ of size
\begin{equation}
|M|\le \frac{n_B-b_2^\prime}2+\frac{r+w_2+w_3+b_2^\prime}2 =\frac{2n+r+b_2+2w_2+2w_3}6.
\label{eq:M}
\end{equation}
In order to determine the upper bound we used the fact that $w_2+w_3\le b_2+b_3$, since in $G_0$, the white vertices of each cycle $C_i$ induce an independent set and relation~\eqref{eq:BW}. Now, we have $w_2+2w_3\le b_2+b_3+w_3\le 2b_2+w_2\le  2b_2+2w_2$, and therefore $r+b_2+2w_2+2w_3\le 3(r+b_2+w_2)$. Plugging the last relation and relation~\eqref{eq:t} into~\eqref{eq:M}, we infer
$$
|M|\le\frac n3+t-2\,.
$$

\item[(B)] \textit{There is just one class with respect to $\sim$ containing all the twelve  pentagons of $F$.} In this case we find a subtree $T^*$ of $F^*$ containing all the vertices corresponding to pentagonal faces of $F$. Let $T$ be the set of edges in $F$ corresponding to the edges of $F^*$ with both endverities in $T^*$. The graph $G=F-T$ is a hexagonal patch. The overall number $t$ of vertices of $T^*$ is at most $11(\sqrt{63n/2}+14)$. We embed $G$ into an infinite tube of the type $(p_1,p_2)$ with $p_1+p_2$ sufficiently large, and follow the same procedure as in the previous case. Observe that here the patch $Q$ (it is only one) has precisely 12 pentagons, and instead~\eqref{eq:BW}, now we have $ b_2+w_2=b_3+w_3-6$. Similarly~\eqref{eq:t}  in this case is $2t-4=r+b_2+w_2+2$. Applying these changes into~\eqref{eq:M}, we can use analogous calculations to prove that this way we find a maximal matching of size at most $n/3+t-4$, what concludes the proof of the theorem.
\end{itemize}
In all the cases we managed to find a maximal matching of size at most $n/3+11(\sqrt{63n/2}+14)-2$, as desired.
\end{proof}


\section{Lower bound on the saturation number}
\label{section:lower}

In this section, we improve the lower bound on the saturation number of fullerene graphs. We show that every maximal matching of a fullerene graph contains at least $ n/3 - 2 $ edges.

\begin{te}
Let $F$ be a fullerene graph on $n$ vertices. Then,
$$
s(F)\ge \frac {n}3-2.
$$
\end{te}

\begin{proof}

Let $M$ be a maximal matching in $F$. Let vertices covered by $M$ be black, edges of $M$ black as well, remaining vertices and edges white.
Let $B$ (resp. $W$) be the set of all black (resp. white) vertices.

In order to prove the theorem, we use the discharging method. We set the charges to vertices and pentagonal faces as follows:
\begin{itemize}
\item Let the initial charge of each black vertex be $3$;
\item Let the initial charge of each white vertex be $-6$; and,
\item Let the initial charge of each pentagonal face be $3$.
 \end{itemize}

We will  prove that the total sum of the charge in the graph $3|B|-6|W|+36$ is non-negative. In other words,
$$
3|B|\ge 2|B|+2|W| -12,
$$
and it implies
$$
|M|=\frac{|B|}2\ge \frac{|B|+|W|-6}3 = \frac{n-6}3.
$$
Hence, we will obtain that the saturation number of $F$ cannot be smaller than the bound
$n/3-2$, if $3|B|-6|W|+36\ge 0$.

Now we will prove the inequality $3|B|-6|W|+36\ge 0$ in order to establish the theorem. First the initial charge is redistributed by the following rule:
\begin{itemize}
\item [$(\mathbf{R1})$.]\textit{Each white vertex sends $-2$ of charge to each adjacent black vertex.}
\end{itemize}
Since $M$ is a maximal matching, $W$ is an independent set in $F$, i.e. no two white vertices are adjacent. After applying $(\mathbf{R1})$, the white vertices have charge zero.

Let $v$ be a black vertex. It is adnacent to at least one black vertex, hence, it is adjacent to at most two white vertices. Let $e_v$ be the black edge incident with $v$, and let $f_v$ be the face incident with $v$, but not with $e_v$.
After having received $0$, $-2$, or $-4$ of charge $(\mathbf{R1})$, according to the number of white neighbors, $v$ has charge $3$, $1$, or $-1$.
\begin{itemize}
\item [$(\mathbf{R2})$.] \textit{Each black vertex $v$ sends all its remaining charge to $f_v$.}
\end{itemize}

All the charge initially present at vertices of $F$ is now at its faces.
The only case when a face was given some negative charge, is when a black vertex $v$ with two white neighbors sends $-1$ of charge to the face $f_v$. Therefore, if a face is incident with at most one white vertex, its charge is non-negative. Moreover, if a pentagon is incident with exactly one white vertex, its charge is at least 3, see Figure \ref{fig:cases51}$(a)$. Similarly, if a hexagon is incident with exactly one white vertex its charge is at least 1.


\begin{figure}[htp!]
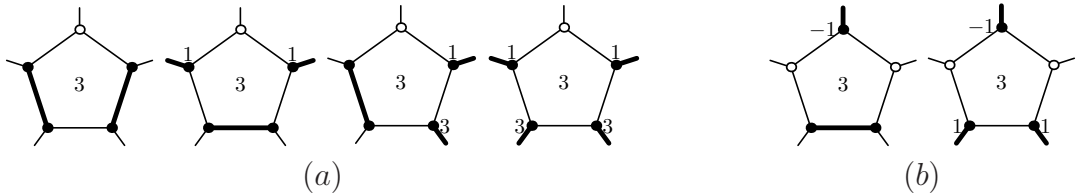

\centerline{\begin{tabular}{cc}
\hspace{.5cm}
\includegraphics[scale=0.8]{cases5.3}
\includegraphics[scale=0.8]{cases5.4}
\includegraphics[scale=0.8]{cases5.5}
\includegraphics[scale=0.8]{cases5.6}
\hspace{.5cm}
&
\hspace{.5cm}
\includegraphics[scale=0.8]{cases5.1}
\includegraphics[scale=0.8]{cases5.2}
\hspace{.5cm}
\\
$(a)$&$(b)$
\end{tabular}
}
\caption{$(a)$ All the possible types of pentagonal faces of $F$ incident to one white vertex. The initial charge of the pentagons  is 3, and its incident black vertices send charge  0, 1 or 3, as shown in the figure. In all the situations the charge of the pentagon after applying $(\mathbf{R2})$ is at least 3. $(b)$ Pentagonal faces of $F$ incident to two white vertices. After applying $(\mathbf{R2})$ the charge of the pentagon is at least $2$.}
\label{fig:cases51}
\end{figure}

\begin{figure}[htp!]
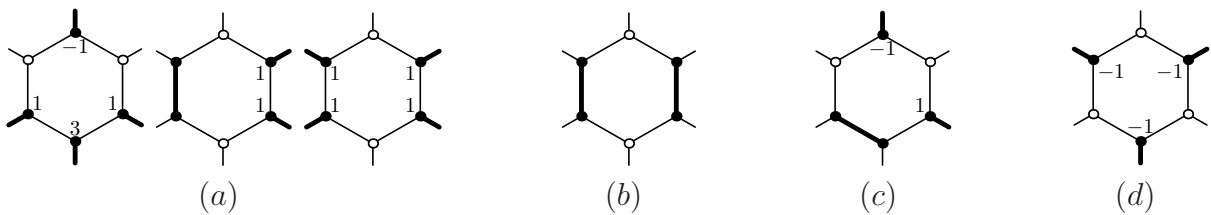

\centerline{\begin{tabular}{cccc}
\hspace{.5cm}
\includegraphics[scale=0.8]{cases.6}
\includegraphics[scale=0.8]{cases.3}
\includegraphics[scale=0.8]{cases.4}\hspace{.5cm}
&
\hspace{.5cm}
\includegraphics[scale=0.8]{cases.2}\hspace{.5cm}
&
\hspace{.5cm}
\includegraphics[scale=0.8]{cases.5}\hspace{.5cm}
&
\hspace{.5cm}
\includegraphics[scale=0.8]{cases.1}\hspace{.5cm}
\\
$(a)$&$(b)$&$(c)$&$(d)$
\end{tabular}
}
\caption{All the possible situations of hexagonal faces of $F$ incident to at least two white vertices. The hexagons in $(a)$ have positive charge; the hexagons in $(b)$ and $(c)$ have charge zero; and the hexagon in $(d)$ has charge $-3$. }
\label{fig:cases62}
\end{figure}


If a pentagonal face is incident with two white vertices, its charge is at least 2, see Figure~\ref{fig:cases51}$(b)$.
If a hexagonal face $h$ is incident with two white vertices as in Figure~\ref{fig:cases62}(a), it will have a positive charge. If the hexagon $h$ is as in Figure~\ref{fig:cases62}(b) or (c), then its charge is zero. Let us call these two types of hexagons of charge zero \emph{neutral} and \emph{transition} faces, respectively. The hexagonal face having three white neighbors (Figure~\ref{fig:cases62}(d)) has charge $-3$. Let call this type of hexagons \emph{bad}. All the other hexagons, as well as all the pentagons, have positive charge, and we call them  \emph{good}.

Let $f$ be a transition hexagonal face. It is incident to two white vertices, two black vertices forming a black edge, and two other black vertices. Let the white vertex adjacent to the two other black vertices be called \emph{incoming}, let the other one (adjacent to an endpoint of the black edge incident with $f$) be called \emph{outgoing}.

The next steps of the discharging are given with the following rules:
\begin{itemize}
\item [$(\mathbf{R3}).$] \textit{Each good face sends charge $1$ to each incident white vertex;}
\item [$(\mathbf{R4}).$] \textit{Each bad hexagonal face sends charge $-1$ to each incident white vertex.}
\item [$(\mathbf{R5}).$] \textit{Each transition hexagonal face sends charge $-1$ to the incoming incident white vertex, and it sends charge $1$ to the outgoing incident white vertex.}
\end{itemize}
It is clear that after applying these three rules there is no negative charge at the faces of $F$. The only elements of the graph that can contain some negative charge are the white vertices incident to a bad or transition hexagon.

Let $v$ be a white vertex that was sent charge $-1$ from a hexagon $h$ by $(\mathbf{R4})$ or $(\mathbf{R5})$. Let $w_1$ and $w_2$ be the black vertices adjacent to $v$ incident with $h$. The black edge incident to $w_i$ is not incident to $h$, $i=1,2$.
Let $u$ be the neighbor of $v$ not incident with $h$. Clearly, $u$ is black. Let $f_1$ and $f_2$ be the two faces incident with $v$ different from $h$. Without loss of generality we may assume that the black edge incident with $u$ is incident with $f_1$. Then $f_1$ is good or neutral, so it does not send negative charge to $v$.

Clearly, $f_2$ cannot be a bad hexagon, nor a neutral one. If $f_2$ is not incident to other white vertex but $v$, it is a good hexagon. If $f_2$ is incident to another white vertex at distance 3 from $v$, it is a good hexagon as well. If $f_2$ is incident to another white vertex at distance 2 from $v$, then it is a transition hexagon, moreover,
$v$ is the outgoing white vertex for $f_2$. In all the cases,
$f_2$ has sent charge 1 to $v$ by $(\mathbf{R3})$ or $(\mathbf{R5})$, see Figure \ref{fig:bad}.

\begin{figure}[htp!]
\centerline{\includegraphics[scale=0.8]{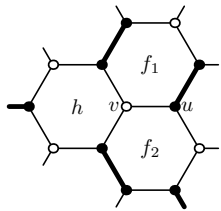}}
\caption{For a white vertex $v$ which receives charge $-1$ from a bad or transition hexagon $h$, there is always another hexagon, which sends positive charge to $v$.}
\label{fig:bad}
\end{figure}

Since there is no negative charge in the graph, the total sum of charge is non-negative, as desired.
\end{proof}

\section{Concluding remarks}

We managed to prove a lower and an upper bound on the saturation number of fullerene graphs, which are asymptotically equal.

The bound proved in Section~\ref{section:lower} turns out to be tight. There are infinitely many fullerene graphs with the saturation number equal to $n/3-2$: for example, a $(8,0)$-nanotube with $3k+1$ rings of hexagons and with caps depicted in Figure \ref{fig:80cap} has $48k+60$ vertices and admits a maximal matching of size $16k+18$. We are aware of other examples, even without adjacent pentagons.

\begin{figure}[htp!]
\centerline{\includegraphics[scale=0.8]{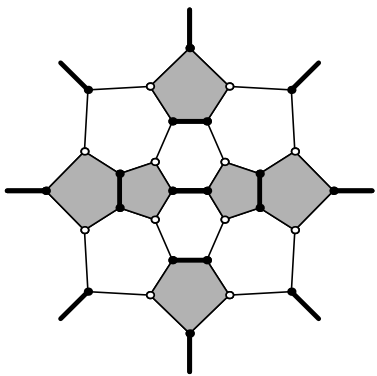}}
\caption{A cap of an $(8,0)$-nanotube with saturation number $n/3-2$.}
\label{fig:80cap}
\end{figure}

Comparing the newly established lower bound $n/3-2$ with the previous bound $3n/10$ we find that a fullerene graph can only admit a maximal matching of size exactly $3n/10$ if it has at most $60$ vertices. Moreover, this can only occur for fullerene graphs having exactly $20$, $30$, $40$, $50$, or $60$ vertices.

\begin{figure}[htp!]
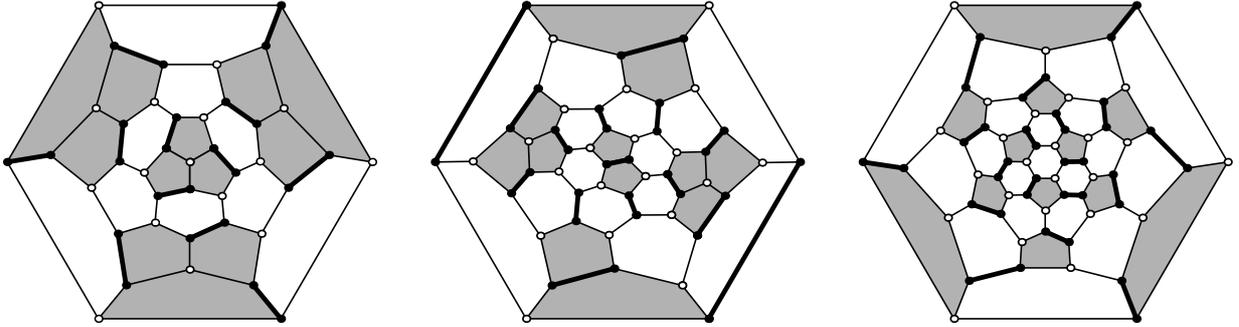

\centerline{\includegraphics[scale=0.8]{ex40.1}
\hfill
\includegraphics[scale=0.8]{ex50.1}
\hfill
\includegraphics[scale=0.8]{ex60.1}}
\caption{Fullerenes on $n= 40, 50, 60$ with saturation number  $3n/10$.}
\label{fig:n/3}
\end{figure}
Since there is only finitely many such graphs and the lists of those graphs are known, we can find easily those that admit a maximal matching of size $3n/10$ by inspecting each of them by a computer check. The dodecahedron (the only fullerene graph on 20 vertices) admits a maximal matching of size 6; none of the three fullerene graphs on 30 vertices does admit a maximal matching of size 9; there is exactly one fullerene graph on 40 and 50 vertices having a maximal matching of size 12 and 15, respectively; there are 7 fullerene graphs on 60 vertices admitting a maximal matching of size 18, including the buckminsterfullerene (the only fullerene graph on 60 vertices without adjacent pentagons).

The question to determine the exact value of the saturation number remains still open. Here we pose a conjecture concerning the problem.

\begin{conjecture}
There is a constant $C$ such that
$$
s(F)\le \frac{n}3 + C
$$
for any fullerene graph $F$ on $n$ vertices.
\end{conjecture}

The problem of finding of minimal independent dominating set is NP-complete~\cite{YG}. This problem is NP-complete even when restricted to planar or bipartite graphs of maximal degree three~\cite{YG}, and remains NP-complete for  planar cubic graphs~\cite{HK}. These results imply the next question.
\begin{problem}
Is the problem to determine the saturation number  for the class of fullerene graphs  NP-complete?
\end{problem}



\begin{thebibliography}{999}
\setlength{\itemsep}{0pt}

\bibitem{ADKLS}
    V.~Andova, T.~Do\v sli\' c, M.~Krnc, B.~Lu\v zar, R.~\v Skrekovski,
    \textit{On the diameter and some related invariants of fullerene graphs},
     MATCH Commun. Math. Comput. Chem.  \textbf{68} (2012),  109--130.

\bibitem{CS}
    M.~Chudnovsky, P.~Seymour,
    \textit{Perfect matching in planar graphs},
    Combinatorica, \textbf{32}, (2012), 403--423.

\bibitem{D}
    S.~M.~Daugherty,
    \textit{Independent Sets and Closed-Shell Independent Sets of Fullerenes},
    Ph.D. thesis, University of Victoria, 2009.




\bibitem{D3}
    T.~Do\v{s}li\'{c},
    \textit{Saturation number of fullerene graphs},
    J. Math. Chem. \textbf{43} (2008), 647--657.



\bibitem{DLS}
   Z.~Dvo\v r\'ak, B.~Lidick\' y, R.~\v Skrekovski,
  \textit{Bipartizing fullerenes},
   European J. Combin. \textbf{33} (2012), 1286--1293.


\bibitem{E}
    J.~Edmonds,
    \textit{Paths, trees, and flowers},
    Canad. J. Math. \textbf{17} (1965), 449--467.

\bibitem{EKKKN}
    L.~Esperet, F.~Kardo\v s, A.~D.~King, D.~Kr\'{a}l, S.~Norine,
    \textit{Exponentially many perfect matchings in cubic graphs},
    Adv. Math. \textbf{227} (2011), 1646--1664.



\bibitem{FL}
    S.~Fajtlowicz, C.~E.~Larson,
    \textit{Graph-Theoretic Independence as a Predictor of Fullerene Stability},
    Chem. Phys. Letters \textbf{377} (2003), 485--490.

\bibitem{FKS}
    L.~Faria, S.~Klein, M.~Stehl\'\i k,
    \textit{Odd cycle transversals and independent sets in fullerene graphs},
    SIAM J. Discrete Math. \textbf{26}, 145--149.

\bibitem{FRFHC}
    P.~W.~Fowler, K.~M.~Rogers, S.~Fajtlowicz, P.~Hansen and G.~Caporossi,
    Facts and conjectures about fullerene graphs: leapfrog, cylindrical
    and Ramanujan fullerenes,
    \textit{In: A. Betten, A. Kohnert, R. Laue and A. Wassermann, Editors, Algebraic Combinatorics and Applications, Springer, Berlin} (2000).

\bibitem{FM}
    P.~W.~Fowler, D.~E.~Manolopoulos,
    \textit{An Atlas of Fullerenes},
    Oxford Univ. Press, Oxford, 1995.


\bibitem{GM}
    B.~Gr\"{u}nbaum, T.~S.~Motzkin,
    \textit{The number of hexagons and the simplicity of geodesicson certain polyhedra},
     Can. J. Math. \textbf{15} (1963), 744--751.


\bibitem{HT}
    C.~H.~Heckman, R.~Thomas,
    \textit{Independent sets in triangle-free cubic planar graphs},
    J. Combin. Theory B \textbf{96} (2006), 253--275.

\bibitem{HK}
    J.~D.~Horton, K.~Kilakos,
    \textit{Minimum edge dominating sets},
    SIAM J. Discret. Math. \textbf{6} (1993), 375--387.

\bibitem{KKMS}
    F.~Kardo\v{s}, D.~Kr\'al', J.~Mi\v{s}kuf, J.-S.~Sereni,
    \textit{Fullerene graphs have exponentially many perfect matchings},
    J. Math. Chem. \textbf{46} (2009) 443--447.

\bibitem{KS}
   F.~Kardo\v{s}, R.~\v{S}krekovski,
    \textit{Cyclic edge-cuts in fullerene graphs},
    J. Math. Chem. \textbf{44} (2007), 121--132.

%



\bibitem{KHOCS}
    H.~W.~Kroto, J.~R.~Heath, S.~C.~O'Brien, R.~F.~Curl,  R.~E.~Smalley,
    \textit{ C60: Buckminsterfullerene},
    Nature \textbf{318} (1985), 162--163.

\bibitem{LP}
    L.~Lov\'{a}sz, M.~D.~Plummer,
    Matching theory,
    \textit{Elsevier Science}, Amsterdam, 1986.

\bibitem{M}
    J.~Malkevitch,
    \textit{Geometrical and combinatorial questions about fullerenes},
    in: P. Hansen, P. Fowler, M. Zheng (Eds.), Discrete Mathematical Chemistry,
    DIMACS Series in Discrete Mathematics and Theoretical Computer Science \textbf{51} (2000), 261--266.

\bibitem{YG}
    M.~Yannakakis, F.~Gavril,
    \textit{Edge Dominating Sets in Graphs},
    SIAM. J. Appl. Math. \textbf{38}(1980), 364--372.

\bibitem{Z}
    M.~Zito, \textit{Small maximal matchings in random graphs},
    Theor. Comput. Sc. \textbf{297} (2003), 487--507.


\end{thebibliography}
 \end{document}